\documentclass{amsart}

\usepackage{amssymb}
\usepackage{longtable}
\usepackage{multirow}

\newcommand{\QQ}{{\mathbb Q}}
\newcommand{\ZZ}{{\mathbb Z}}
\newcommand{\NN}{{\mathbb N}}

\newcommand{\cC}{{\mathcal C}}
\newcommand{\cD}{{\mathcal D}}
\newcommand{\cO}{{\mathcal O}}
\newcommand{\cL}{{\mathcal L}}
\newcommand{\cA}{{\mathcal A}}
\newcommand{\cP}{{\mathcal P}}

\newtheorem{theorem}{Theorem}
\newtheorem{proposition}[theorem]{Proposition}

\theoremstyle{definition}

\theoremstyle{remark}
\newtheorem{remark}{Remark}

\theoremstyle{conjecture}
\newtheorem{conjecture}{Conjecture}

\numberwithin{equation}{section}

\begin{document}

\title{Markoff--Rosenberger triples in arithmetic progression}

\author{Enrique Gonz\'alez--Jim\'enez}
\address{Universidad Aut{\'o}noma de Madrid, Departamento de Matem{\'a}ticas and Instituto de Ciencias Matem{\'a}ticas (ICMat), Madrid, Spain}
\email{enrique.gonzalez.jimenez@uam.es}
\urladdr{http://www.uam.es/enrique.gonzalez.jimenez}
\author{Jos\'e M. Tornero}
\address{Departamento de \'Algebra, Universidad de Sevilla. P.O. 1160. 41080 Sevilla, Spain.}
\email{tornero@us.es}
\thanks{The first author was partially  supported by the grant MTM2009--07291. The second author was partially  supported by the grants  FQM--218 and P08--FQM--03894, FSE and FEDER (EU)}

\begin{abstract}
We study the solutions of the Rosenberg--Markoff equation $ax^2+by^2+cz^2 = dxyz$ (a generalization of the well--known Markoff equation). We specifically focus on looking for solutions in arithmetic progression that lie in the ring of integers of a number field. With the help of previous work by Alvanos and Poulakis, we give a complete decision algorithm, which allows us to prove finiteness results concerning these particular solutions. Finally, some extensive computations are presented regarding two particular cases: the generalized Markoff equation $x^2+y^2+z^2 = dxyz$ over quadratic fields and the classic Markoff equation $x^2+y^2+z^2 = 3xyz$ over an arbitrary number field.
\end{abstract}

\maketitle


\section{Variations on the Markoff equation}

The Markoff equation is the Diophantine equation
$$
x^2+y^2+z^2 = 3xyz; \quad x,y,z \in \ZZ_+;
$$
which was studied first by Markoff in \cite{Markoff1, Markoff2}. In those papers, many interesting properties related to the solutions of this equation were proved. Among other things, Markoff showed there were infinitely many solutions (so--called {\em Markoff triples}), gave a procedure to construct new solutions from old ones and proved that, in fact, all integral solutions could be constructed from one {\em fundamental solution} $(1,1,1)$.

Since then, the Markoff equation and its solutions have been object of intense research. Remarkably, Frobenius, while studying the Markoff equation over Gaussian integers \cite{Frobenius}, noticed that, for a given ordered solution $x \leq y \leq z$, there was no other ordered solution $x' \leq y' \leq z$. This conjecture, widely known as the Frobenius unicity conjecture, has remained open since, although some important partial results have been settled \cite{Baragar,BaragarRM,Button1, Schmutz,Button2}.

Our work started with the Markoff equation, but from a different point of view. Rather than looking at the Frobenius conjecture, we decided to focus on looking for Markoff triples with some extra structure: more precisely, those Markoff triples which are in arithmetic progression (a.p. in what follows). However, in our study, it became apparent that a generalization could be considered instead of the original Markoff equation. From the many extended versions of the Markoff equation, two have been by far the most studied in the literature.

First, the Hurwitz (or Markoff--Hurwitz) equation \cite{Hurwitz}, given by
$$
x_1^2+x_2^2+...+x_n^2 = a x_1x_2...x_n,
$$
for which Hurwitz himself proved that all solutions could be constructed from a set of easy mappings acting on a finite set of particular solutions.

The other succesfully studied generalization mentioned above, and the one who will center much of our work, is the so--called Markoff--Rosenberger equation \cite{Rosenberger}:
$$
ax^2+by^2+cz^2 = dxyz.
$$

Note that, in \cite{Rosenberger} (and in all subsequent articles concerning the Rosenberger generalization) it is further required that $a|d$, $b|d$, $c|d$. This comes from Rosenberger's original motivation, related to binary forms and Fuchsian groups. We will not assume this for most of our work. As far as we know, there are no results in the literature concerning the Markoff--Rosenberger equation with no conditions.

Let us consider then the equation 
$$
ax^2+by^2+cz^2 = dxyz; 
$$
and assume we have a solution in a.p. which may be written in the form
$$
x = X, \quad y = X + Y, \quad z = X + 2 Y.
$$

The equation then becomes
$$
dX^3 + 3dX^2Y + 2dXY^2 - (a+b+c)X^2 - (2b+4c)XY - (b+4c)Y^2=0,
$$
which is the equation of a cubic curve.

For what follows we will be taking homogeneous coordinates $[X:Y:Z]$ and considering a projective closure of our curve:
$$
\cC:dX^3 + 3dX^2Y + 2dXY^2 - (a+b+c)X^2Z - (2b+4c)XYZ - (b+4c)Y^2Z=0,
$$ 
which has, regardless of $(a,b,c,d)$, a singular point (actually a node) at $[0:0:1]$ and three points at infinity: $[0:1:0]$, $[1:-1:0]$ and $[2:-1:0]$. Our aim was finding all integral affine points of $\cC$.

This situation makes particularly simple the use of the \verb|INTEGRAL-POINTS| algorithm by Alvanos and Poulakis \cite{AlvanosPoulakis}. What makes this algorithm more remarkable is the fact that it works for arbitrary number fields computing affine points whose coordinates lie in the corresponding ring of integers. So, from Markoff integral triples we had come to Markoff--Rosenberger triples which lie in the ring of integers of a number field (always in a.p.).


In the third section we will develop an algorithm to compute all the Markoff--Rosenberger triples over a number field $K$ for fixed $a,b,c,d\in\mathcal{O}_K$, where $\mathcal{O}_K$ denotes the ring of integers of $K$. In particular, this algorithm allows us to give a characterization of the values $a,b,c,d$ such that there exists a non--trivial triple in a.p. over $\mathcal{O}_K$ (see Proposition \ref{prop1}).  In the fourth section we show several finiteness results related to our problem. In the last section we will deal first with the generalized Markoff equation, obtaining theoretical results over the rationals and over imaginary quadratic fields.  Moreover, we will include some extensive computations we have performed. These computations lead to a well--supported conjecture over quadratic fields that might encourage future research in this area. Finally, we will work with the classic Markoff equation but over arbitrary number fields with bounded discriminant.

But first, as the \verb|INTEGRAL-POINTS| algorithm will be heavily used in what follows, we will recall its steps, for the convenience of the reader. The proofs concerning correctness and termination, as well as many other interesting features can be found in the original reference \cite{AlvanosPoulakis}.

\section{An algorithmic short trip}

Alvanos and Poulakis developed in \cite{AlvanosPoulakis} a very polished algorithm for the computation of the set of affine points on a genus zero curve whose coordinates can be chosen to lie in a ring of algebraic integers. The version presented here is actually the so--called (by the authors) \verb|INTEGRAL-POINTS3|, where \verb|3| refers to the number of points at infinity. The precise algorithm goes as follows: fix a number field $K$ and we are given a curve, say $\cC$, defined by an affine equation 
$$
F(X,Y)=0, \mbox{ where } F(X,Y) \in K[X,Y], \; \deg(F)=\deg_Y(F)=N;
$$ 
verifying that $\cC$ has exactly three smooth points at infinity $\{V_1,V_2,V_3 \}$. We want to compute the affine points of $\cC(\cO_K)$.

\noindent {\bf Step 1.} Compute the singular points of $\cC$ and save those in $\cC(\cO_K)$.

\noindent {\bf Step 2.} Find number fields $M_1$, $M_2$ such that $K \subset M_i \subset \overline{K}$ and polynomials
$$
a_i \in M_i[X,Y], \quad b_i \in M_i[X];
$$
such that $\deg_Y (a_i) < N$ and 
$$
f_i(X,Y):=a_i(X,Y)/b_i(X)\in \cL(V_3 - V_i),
$$
where $\cL(V_3 - V_i)$ denotes the Riemann-Roch space of the divisor $V_3 - V_i$.

\noindent {\bf Step 3.} Compute $\alpha_i, \beta_i \in \cO_{M_i}$ such that $\alpha_i f_i$ and $\beta_i/f_i$ are integral over $\cO_{M_i}[X]$. This step is carried out by an algorithm called \verb|DENOMINATORS| which is presented before in \cite{AlvanosPoulakis}.

\noindent {\bf Step 4.} Determine maximal sets $A_i \subset \cO_{M_i}$ of elements which are not pairwise associate and such that its norm divides that of $\alpha_i\beta_i$.

\noindent {\bf Step 5.} Let $M$ be the normal closure of the composition of $M_1$ and $M_2$. Solve then, in $M$, the equation
$$
c_1f_1+c_2f_2 = 1.
$$

\noindent {\bf Step 6.} For every $(k_1, k_2) \in A_1 \times A_2$ determine the (finite) solution set $S(k_1,k_2)$ of the unit equation
$$
\left( \frac{c_1k_1}{\alpha_1} \right) U_1 + \left( \frac{c_2k_2}{\alpha_2} \right) U_2 = 1.
$$
The finiteness of this set of solutions is a well-known fact which goes back to Siegel \cite{Siegel}. A more recent account on how to actually compute this set of solutions can be found, for instance in \cite{Smart} or the most efficient algorithm of Wildanger \cite{Wildanger}.

\noindent {\bf Step 7.} For any $(k_1,k_2) \in A_1 \times A_2$ and $(u_1,u_2) \in S(k_1,k_2)$ compute the resultant
$$
R_{(k_1,u_1)} (X) = \mbox{Res}_Y \left( F(X,Y), \alpha_1a_1 (X,Y) - k_1u_1b_1(X) \right),
$$
and determine, the set $S$ of solutions in $\cO_K$ for some $R_{(k_1,u_1)}$.

\noindent {\bf Step 8.} For any $v \in S$ compute the possible pairs $(v,w) \in \cC(\cO_K)$.

\noindent {\bf Step 9.} The affine points of $\cC(\cO_K)$ are those computed in Step 1 and those computed in Step 8.

\section{Integral points on curves at work}\label{sec3}

First we must put our curve in a suitable form for \verb|INTEGRAL-POINTS3|, being careful in order to preserve integral points. This can be achieved making the change
$$
X \longmapsto Y - X, \quad\quad Y \longmapsto X, \quad\quad Z \longmapsto Z;
$$
which preserves not only integral projective points but also points at infinity (this is important regarding the algorithm), hence obtaining 
$$
\cD: dY^3-(a+b+c)Y^2Z-dX^2Y+2(a-c)XYZ-(a+c)X^2Z=0.
$$

Our new curve has still the same singular point $[0:0:1]$ (a node) and three points at infinity: 
$$
P_1 = [1:1:0], \quad\quad P_2=[1:-1:0], \quad\quad P_3 =[1:0:0],
$$
all of them smooth. Our aim is then to find all affine points in $\cD(\cO_K)$.

Moving on to Step 2 we must find then generators for $\cL(P_3-P_1)$ and $\cL(P_3-P_2)$. These are given by easy calculation (\'a la Riemann--Roch, so to say):
$$
f_1(X,Y) = \frac{dY^2 - (a+b+c)Y - dX^2 + (a-3c)X}{(a+c)X} \in \cL(P_3-P_1)
$$
$$
f_2(X,Y) = \frac{-dY^2 + (a+b+c)Y + dX^2 + (-3a+c)X}{(a+c)X} \in \cL(P_3-P_2)
$$

Next (Step 3), consider
$$
\alpha_1 = \alpha_2 = a+c, \quad \beta_1=b+4c, \quad \beta_2=b+4a,
$$
and define 
$$
\begin{array}{l}
R_1(T) = T^3 + (dX+a+5c)T^2 + (a+c)(2dX+b+8c)T + (a+c)^2(b+4c), \\
R_2(T) = T^3 + (-dX+5a+c)T^2 + (a+c)(-2dX+8a+b)T + (a+c)^2(b+4a), \\
S_1(T) = T^3 + (2dX+b+8c)T^2 + (b+4c)(dX+a+5c)T + (a+c)(b+4c)^2, \\
S_2(T) = T^3 + (-2dX+8a+b)T^2 + (4a+b)(-dX+5a+c)T + (a+c)(b+4a)^2, \\
\end{array}
$$
for which 
$$
R_i \left( \alpha_if_i(X,Y) \right) = S_i \left( \frac{\beta_i}{f_i(X,Y)} \right) = 0, \quad i=1,2.
$$

Now we consider (Step 4) the following sets:
\begin{equation}\label{Ais}
A_i = \left\{ k_i \in \cO_K \; | \; \mathcal{N}_{K}(k_i) \mbox{ divides } \mathcal{N}_{K}(\alpha_i \beta_i) \right\} / \sim,
\end{equation}
where $\sim$ denotes we are actually interested in the equivalence class (modulo associated elements) and $\mathcal{N}_{K}$ denotes the absolute norm map. Note that this step depends on $K$.

As for Step 5 is concerned, we have $f_1+f_2 = -2$.

Now we must consider (again depending on $K$), for every $(k_1,k_2) \in (A_1,A_2)$, the finite set $S_{k_1,k_2}$ of solutions to the unit equation
\begin{equation}\label{unit}
k_1u_1+k_2u_2 = -2(a+c).
\end{equation}

In Step 7, for any $(k_1,k_2) \in (A_1,A_2)$ and for any $(u_1,u_2) \in S_{k_1,k_2}$, we have
\begin{eqnarray*}
R_{k_1,u_1} (X) &=& \mbox{Res}_Y \left( F(X,Y), (a+c)a_1 (X,Y) - k_1u_1b_1(X) \right) \\
&=& -(a+c)^3d^2X^3 \left( \displaystyle X - z_{k_1,u_1} \right),
\end{eqnarray*}
where
$$
z_{k_1,u_1} = \frac{(a+c+k_1u_1)(4ck_1u_1+k_1^2u_1^2+(a+c)(b+4c))}{u_1u_2k_1k_2d}.
$$
As the $X$--coordinates of all affine points in $\cD(\cO_K)$ appear as roots (in $\cO_K$, of course) of some $R_{k_1,u_1} (X)$, we have two candidates $X=0$ and $X =z_{k_1,u_1} $.

Next step is computing the possible points. For $X=0$ we have
$$
\left( 0, \frac{a+b+c}{d} \right) \in \cD(K).
$$

For $X=z_{k_1,u_1}$ we have 
$$
\left( z_{k_1,u_1}, \frac{a+c}{a+c+k_1u_1} z_{k_1,u_1} \right) \in \cD(K),
$$
and, remarkably,
$$
\left( z_{k_1,u_1}, \frac{\left[ (a+c)(b+4c)+(2a+b+2c)k_1u_1 \pm \sqrt{\Delta} \right] (a+c+k_1u_1)}{-2u_1u_2k_1k_2d} \right) \in \cD \left( L \right),
$$
where $L=K \left( \sqrt{\Delta} \right)$ and
\begin{eqnarray*}
\Delta &=& 4(k_1u_1)^4 + 8(a+3c)(k_1u_1)^3 + (4a^2+b^2+8ab+56ac+8bc+52c^2)(k_1u_1)^2\\
&& \quad +2(a+c)(b+4c)(6a+b+6c)k_1u_1+(a+c)^2(b+4c)^2
\end{eqnarray*}

Hence, we have the following characterization.

\begin{proposition}\label{prop1}
Let $K$ be a number field and $a,b,c,d\in \mathcal{O}_K$. Then, there exists a non--trivial solution in a.p. over $\mathcal{O}_K$ for the Markoff--Rosenberger equation $ax^2+by^2+cz^2=dxyz$ if and only if one of the following conditions hold:
\begin{enumerate}
\item[(a)] $d\,|\,(a+b+c)$.
\item[(b)] $d\,|\,(a+c+k_1u_1)(k_1^2u_1^2+4ck_1u_1+(a+c)(b+4c))$, for some $(k_1,k_2) \in (A_1,A_2)$ and some $(u_1,u_2) \in S_{k_1,k_2}$.
\end{enumerate}
\end{proposition}

\section{Finiteness results}

We will write, for a given ring $R$, $R/R^*$ for the set of elements of $R$ up to multiplication by a unit (as customary), and $R/R^2$ for the set of elements of $R$ with no square root in $R$.

Let us call from now on, for a given number field $K$ and given $a,b,c,d \in \mathcal{O}_K$,
$$
\cA\cP_{(a,b,c,d)} (K) = \left\{ \mbox{$\cO_K$--solutions in a.p. to $ax^2+by^2+cz^2 = dxyz$ } \right\},
$$
where obviously we always have the trivial solution $(0,0,0)$. Remember that we are disregarding the Rosenberger conditions $a|d$, $b|d$, $c|d$. 

First, mind that we have set--up a bijection between $\cA\cP_{(a,b,c,d)} (K)$ and the affine points of $\cD(\cO_K)$, which we already know to be a finite set \cite{Maillet,Lang} (because $\cD$ has three points at infinity). 

\begin{theorem}
Let $K$ be a number field and  $a,b,c,d \in \mathcal{O}_K$. Then $\cA\cP_{(a,b,c,d)} (K)$ is finite.
\end{theorem}

The following results follow directly from the characterization given on Proposition \ref{prop1} above and deep results of Corvaja and Zannier \cite{CZ}:

\begin{theorem}
 Let $K$ be a number field and  $a,b,c \in \mathcal{O}_K$. Then

\begin{enumerate} 
\item We have
$$
\# \left\{ d \in\mathcal{O}_K / \cO_K^* \; | \; \cA\cP_{(a,b,c,d)} (K) \neq \{(0,0,0)\} \right\}< \infty.
$$
\item If $\Delta \in \mathcal{O}_K / \cO_K^2$, then
$$
\# \left\{ d \in \mathcal{O}_K / \cO_K^* \; | \;  \cA\cP_{(a,b,c,d)} (K) \subsetneq \cA\cP_{(a,b,c,d)} \left( K \left( \sqrt{\Delta} \right) \right)  \right\} < \infty.
$$
\item Let $L/K$ be a finite algebraic extension. Then
$$
\# \left\{ d \in \mathcal{O}_K / \cO_K^* \; | \; \cA\cP_{(a,b,c,d)} (K) \subsetneq \cA\cP_{(a,b,c,d)} (L)  \right\} < \infty.
$$
\item We have
$$
\# \left\{ (d,\Delta)\in \mathcal{O}_K/ \cO_K^*\times \mathcal{O}_K / \cO_K^2 \; | \; \cA\cP_{(a,b,c,d)} (K) \subsetneq \cA\cP_{(a,b,c,d)} \left( K \left( \sqrt{\Delta} \right) \right)  \right\} < \infty.
$$
\item Let $D \in \ZZ_{>0}$ and denote by $\mathcal{A}_{D}(K)$ the set of algebraic extensions of degree $D$ of $K$ up to isomorphism. Then
$$
\# \left\{ (d,L) \in \mathcal{O}_K / \cO_K^*\times \mathcal{A}_{D}(K)\; | \; \; \cA\cP_{(a,b,c,d)} (K) \subsetneq \cA\cP_{(a,b,c,d)} (L)  \right\} < \infty.
$$
Moreover, we have 
$$
\# \left[ \bigcup_{L\in\mathcal{A}_{D}(K)} \left( \bigcup_{d \in \mathcal{O}_K/ \cO_K^*} \cA\cP_{(a,b,c,d)} (L) \right) \right] < \infty.
$$
\end{enumerate}
\end{theorem}

\begin{proof}
The first three statements come directly from Proposition \ref{prop1}. As for the remaining cases, let us recall Corollary 1 from \cite{CZ} for the case $\mathcal{O}_K$ which stated that if $\cC$ is a plane curve with three or more points at infinity, $K$ a number field  and $D \in \ZZ_{>0}$ then if $L$ runs through all algebraic extension of degree $D$ of $K$:
$$
\# \left( \bigcup_{[L:K] \leq D} \cC(\mathcal{O}_L ) \right) < \infty.
$$
Applying this above result to $\mathcal{D}$, we get the desired statements.
\end{proof}

\begin{remark} Rosenberger proved that, with the extra conditions we mentioned in the first section, the only equations with non--trivial integral solutions were those given by
$$
(a,b,c,d) \in \left\{ (1,1,1,1), \; (1,1,1,3), \; (1,1,2,2), \; (1,1,2,4), \; (1,2,3,6), \; (1,1,5,5) \right\}.
$$

All of them have solutions in a.p. The first five of them verify condition (a) on the characterization given at Proposition \ref{prop1}. The last one
$$
x^2 + y^2 + 5z^2 = 5xyz,
$$
does not, but it verifies the second condition (two solution in a.p. being $(-3,-1,1)$ and $(-7,-1,5)$).
\end{remark}

\section{Computational results} 
We have implemented in \verb|Magma| \cite{magma} the algorithm developed on section \ref{sec3}. Note that, given a number field there are only three major problems to solve: to compute the sets $A_i$ given on (\ref{Ais}), to solve the unit equation given by (\ref{unit}) and to determine if an algebraic number is integral; they are sorted out by the \verb|Magma| functions \verb|NormEquation|, \verb|UnitEquation| and  \verb|IsIntegral|, respectively.

\
 
Our original goal, the study of Markoff triples in a.p. can by now be easily achieved.

\

\begin{theorem} 
\
\begin{itemize}
\item $\mathcal{AP}_{(1,1,1,1)} (\QQ) = \{(0,0,0),(3,3,3),(-15,-6,3),(3,-6,-15)\}$.
\item $\mathcal{AP}_{(1,1,1,3)} (\QQ) = \{(0,0,0),(1,1,1),(-5,-2,1),(1,-2,-5)\}$.
\item If $d\ne 1,3$, then $\mathcal{AP}_{(1,1,1,d)} (\QQ) =\{(0,0,0)\}$.
\end{itemize}
\end{theorem}

Actually, Markoff himself proved \cite{Markoff1,Markoff2} that, if $d \neq 1,3$, the equation has no integer solutions, so the general case is trivial. For the cases $d=1,3$ we run our algorithm to obtain the above results. Note that a simpler algorithm is available thanks to Poulakis and Voskos \cite{PoulakisVoskos} when the base field is $\QQ$.

So, as a direct application of the algorithms explained above, we tried next to study exhaustively the generalized Markoff equation
$$
x^2+y^2+z^2 = dxyz, \quad d \in \ZZ
$$
looking for solution in a.p. over arbitrary quadratic fields. 

When we move to this, more general, case, we find a groundbreaking paper by Silverman \cite{Silverman}. In that paper, known facts from the integral case (how to obtain all solutions, number of points of bounded height and so on) are carefully generalized for the imaginary quadratic case. For this case we obtain the following result.

\begin{theorem}
Let $D\in \ZZ/\ZZ^2$, $D<0$ and $i=\sqrt{-1}$. Then   
$$
\begin{array}{lcl}
\bullet\,\, \mathcal{AP}_{(1,1,1,1)} (\QQ(i)) &=&\left\{(0,0,0),(3,3,3),(-15,-6,3),(3,-6,-15), \right. \\[1mm]
& & (2, \pm 2i + 2, \pm 4i + 2),  (\pm i + 2, 2, \mp i + 2),  \\[1mm]
& &(\pm i + 2, \pm 2i - 1, \pm 3i - 4) ,(\pm 2i - 1, -1, \mp 2i - 1), \\[1mm]
& & \left.(\pm 4i + 2, 2i + 2, 2),  (\pm 3i - 4, \pm 2i - 1, \pm i + 2) \right\}\ \\[1mm]
\bullet\,\, \mathcal{AP}_{(1,1,1,2)} (\QQ(i)) &=& \left\{  (0, 0, 0), (\pm 2i + 1, \pm i + 1, 1), (1, \pm i + 1, \pm 2i + 1)\right\}. \\[1mm]
\bullet\,\, \mathcal{AP}_{(1,1,1,d)} (\QQ(\sqrt{D}))  &=&\mathcal{AP}_{(1,1,1,d)} (\QQ), \mbox{ if }(D,d)\ne (-1,1),(-1,2).
\end{array}
$$
\end{theorem}

\begin{proof} Let us put $a=b=c=1$ in the algorithm. Under these hypothesis, the sets $A_i$ given on (\ref{Ais}) satisfy $A_1=A_2$ and, if $k_i\in A_i$ then it must hold $\mathcal{N}_{K}(k_i) \, | \, 100$. Now we have $k_i=u+v\sqrt{D}$ (resp. $k_i=(u+v\sqrt{D})/2$) if $D\not\equiv 1 \,(\mbox{mod $4$})$ (resp. $D\equiv 1\,(\mbox{mod $4$})$). 

$-$ If $v\ne 0$ then $|D|\le 100$ (resp.  $|D|\le 400$). Hence we just perform the algorithm for all imaginary quadratic fields up to this bound. For every $(k_1,k_2)\in A_1\times A_2$ we compute the unit equation $k_1u_1+k_2u_2=-4$ (see equation (\ref{unit})). For every solution of this, we make $w_{k_1,u_1}=z_{k_1,u_1} d$. Please note that up to now our arguments are independent of $d$. As we need $z_{k_1,u_1}\in \mathcal{O}_K$, in particular $w_{k_1,u_1}\in \mathcal{O}_K$. So with this condition we can forget about those elements where $\mathcal{N}_{K}(w_{k_1,u_1})\notin\ZZ$. Next we factor $\mathcal{N}_{K}(w_{k_1,u_1})=r\cdot s^2$ with $r \in \ZZ / \ZZ^2$ and $(r,s)=1$. This way, we obtain the candidate shortlist $d=s$. Now we run the algorithm to compute $\mathcal{AP}_{(1,1,1,d)} (\QQ(\sqrt{D}))$ within a finite set of pairs $(d,D)$. For all these, we get 
$$
\mathcal{AP}_{(1,1,1,d)} (\QQ(\sqrt{-D})) \ne \mathcal{AP}_{(1,1,1,d)} (\QQ) \mbox{ iff } (D,d)= (-1,1),(-1,2).
$$

$-$ If $v=0$ we have $k_i\in\{1,2,5,10\}$ for all $D$. As we have previously dealt with the cases $|D|\le 400$ or $|D|\le 100$ (depending on $D \mbox{ mod } 4$) we have to concern about $|D|>400$ or $|D|>100$, and then $A_i=\{1,2,5,10\}$ and the units in the ring of integers of $\QQ(\sqrt{-D})$ are just $\pm 1$. The case now parallels the rational one and $\mathcal{AP}_{(1,1,1,d)} (\QQ(\sqrt{-D})) =\mathcal{AP}_{(1,1,1,d)} (\QQ)$.

Please note that Silverman \cite[Theorem 0.1]{Silverman} proves (among many other things) that if $d \geq 3$, the only quadratic imaginary field for which there are Markoff triples at all is $D=-1$. So we might have proved the theorem looking only at cases $(d,D)\in \{ (1,D),(2,D),(d,-1)\}$ for any $D$ and $d\ge 3$. In any case, for these cases we would still need the previous arguments. So, outstanding as Silverman's result is, it also is of little use for the proof of this result.
\end{proof}

For the case of real quadratic case, we have put our algorithm to work in $K = \QQ(\sqrt{D})$, for $|d| \leq 10^3$ and $1<\Delta_K \leq 10^4$ ($\Delta_K$ being the discriminant of $\mathcal{O}_K$). After these computations, we can state the following conjecture.

\begin{conjecture} 
$$
\# \left[ \bigcup_{(d,D) \in \ZZ/\ZZ^*\times \ZZ/\ZZ^2}  \cA\cP_{(1,1,1,d)} \left(\QQ( \sqrt{D} )\right) \right] = 178.
$$
Moreover, the following table showes all the triples in a.p. over quadratic fields:
{\footnotesize
\begin{longtable}{|c|c|l|}
\hline
$d$ & $D$ & Markoff triples in a.p. over $\mathbb{Z}[\alpha]=\mathcal{O}_{\mathbb{Q}(\sqrt{D})}$ not in $\mathbb{Z}$: first term and difference \\
\hline
\multirow{23}{*}{$1$}  & \multirow{3}{*}{$-1$}  & $(\alpha + 2, \alpha - 3)$, $(-\alpha + 2, \alpha)$, $(2\alpha - 1, -2\alpha)$, $(\alpha + 2, -\alpha)$, $(-4\alpha + 2, 2\alpha)$\\ 
 & &  $(-3\alpha - 4, \alpha + 3)$, $(-2\alpha - 1, 2\alpha)$, $(4\alpha + 2, -2\alpha)$, $(2, 2\alpha)$, $(2, -2\alpha)$\\
 & & $(3\alpha - 4, -\alpha + 3)$, $(-\alpha + 2, -\alpha - 3)$ \\
\cline{2-3}
 &  \multirow{3}{*}{$2$} & $(4\alpha + 8, -4\alpha)$, $(11\alpha - 13, -23\alpha - 6)$, $(-4\alpha + 8, 4\alpha)$, $(-7\alpha - 7, 7\alpha)$\\
 & & $(-11\alpha - 13, 23\alpha - 6)$, $(-35\alpha - 25, 23\alpha + 6)$, $(35\alpha - 25, -23\alpha + 6)$\\
 & &  $(7\alpha - 7, -7\alpha)$ \\
\cline{2-3}
 &  $3$ & $(9\alpha + 18, -9\alpha)$, $(-9\alpha + 18, 9\alpha)$ \\
\cline{2-3}
 &  \multirow{7}{*}{$5$} & $(-35\alpha - 5, 22\alpha + 11)$, $(22\alpha - 11, -22\alpha - 11)$, $(-7\alpha - 6, 4\alpha + 5)$\\
 & &  $(-9\alpha + 8, 22\alpha + 11)$, $(35\alpha + 30, -22\alpha - 11)$, $(-7\alpha -9, 2\alpha + 7)$,\\
 & & $(\alpha + 4, -5\alpha + 2)$, $(-3\alpha + 5, -2\alpha - 7)$, $(-\alpha + 3, 4\alpha - 1)$\\
 & & $(7\alpha + 1, -4\alpha + 1)$, $(-\alpha + 3, 5\alpha + 7)$, $(14\alpha - 3, -10\alpha - 2)$\\
 & &  $(-14\alpha - 17, 10\alpha + 8)$, $(9\alpha + 17, -22\alpha - 11)$, $(-22\alpha - 33, 22\alpha + 11)$\\
 & &  $(3\alpha + 8, 2\alpha - 5)$, $(7\alpha - 2, -2\alpha + 5)$, $(-9\alpha+ 8, 5\alpha - 2)$, $(9\alpha + 17, -5\alpha - 7)$\\
 & &  $(6\alpha - 1, -10\alpha - 8)$, $(-6\alpha - 7, 10\alpha + 2)$, $(\alpha + 4, -4\alpha - 5)$ \\
\cline{2-3}
 &  $6$ & $(-6\alpha - 12, 6\alpha)$, $(6\alpha - 12, -6\alpha)$, $(3\alpha - 3, -3\alpha)$, $(-3\alpha - 3, 3\alpha)$ \\
\cline{2-3}
 &  $11$ & $(-2\alpha - 4, 4\alpha - 6)$, $(6\alpha - 16, -4\alpha + 6)$, $(-6\alpha - 16, 4\alpha + 6)$, $(2\alpha - 4, -4\alpha - 6)$ \\
\cline{2-3}
 &  $14$ & $(-2\alpha - 4, 2\alpha)$, $(2\alpha - 4, -2\alpha)$ \\
\cline{2-3}
 &  $17$ & $(4\alpha + 13, \alpha - 10)$, $(-4\alpha + 9, -\alpha - 11)$, $(6\alpha - 7, -\alpha + 10)$, $(-6\alpha - 13, \alpha + 11)$ \\
\cline{2-3}
 &  $21$ & $(3\alpha - 3, 9)$, $(-3\alpha + 12, -9)$, $(-3\alpha - 6, 9)$, $(3\alpha + 15, -9)$ \\
\cline{2-3}
 & \multirow{2}{*}{$29$}& $(-11\alpha - 32, 7\alpha + 14)$, $(11\alpha - 21, -7\alpha + 7)$, $(3\alpha - 4, -2\alpha + 5)$, \\
 & & $(-3\alpha - 7, 2\alpha + 7)$, $(\alpha + 7, -2\alpha - 7)$, $(-3\alpha - 7, 
7\alpha - 7)$, $(-\alpha + 6, 2\alpha - 5)$, \\
& & $(3\alpha - 4, -7\alpha - 14)$ \\
\cline{2-3}
 &  \multirow{2}{*}{$41$} & $(-4\alpha + 15, \alpha - 10)$, $(2\alpha - 3, -\alpha + 4)$, $(5, -\alpha - 5)$, $(2\alpha - 3, \alpha + 11)$\\
 & & $(-2\alpha - 5, \alpha + 5)$,  $(5, \alpha - 4)$, $(4\alpha + 19, -\alpha - 11)$, $(-2\alpha - 5, -\alpha + 10)$ \\
\hline
\multirow{5}{*}{$2$}  &  $-1$ & $(1, -\alpha)$, $(2\alpha + 1, -\alpha)$, $(1, \alpha)$, $(-2\alpha + 1, \alpha)$ \\
\cline{2-3}
 &  $2$ & $(-2\alpha + 4, 2\alpha)$, $(2\alpha + 4, -2\alpha)$ \\
 \cline{2-3}
 &  $6$ & $(3\alpha - 6, -3\alpha)$, $(-3\alpha - 6, 3\alpha)$ \\
 \cline{2-3}
 &  $11$ & $(-3\alpha - 8, 2\alpha + 3)$, $(3\alpha - 8, -2\alpha + 3)$, $(\alpha - 2, -2\alpha - 3)$, $(-\alpha - 2, 2\alpha - 3)$ \\
 \cline{2-3}
 &  $14$ & $(-\alpha - 2, \alpha)$, $(\alpha - 2, -\alpha)$ \\
\hline
\multirow{3}{*}{$3$}   &  3 & $(3\alpha + 6, -3\alpha)$, $(-3\alpha + 6, 3\alpha)$ \\
\cline{2-3}
 &  $6$ & $(\alpha - 1, -\alpha)$, $(-2\alpha - 4, 2\alpha)$, $(-\alpha - 1, \alpha)$, $(2\alpha - 4, -2\alpha)$ \\
 \cline{2-3}
 & $ 21$ & $(\alpha + 5, -3)$, $(\alpha - 1, 3)$, $(-\alpha + 4, -3)$, $(-\alpha - 2, 3)$ \\
\hline
$4$ &  $2$ & $(\alpha + 2, -\alpha)$, $(-\alpha + 2, \alpha)$ \\
\hline
$6$ &  $6$ & $(-\alpha - 2, \alpha)$, $(\alpha - 2, -\alpha)$ \\
\hline
$7$ &  $2$ & $(\alpha - 1, -\alpha)$, $(-\alpha - 1, \alpha)$ \\
\hline
$9$ &  $3$ & $(\alpha + 2, -\alpha)$, $(-\alpha + 2, \alpha)$ \\
\hline
$11$ &  $5$ & $(-2\alpha - 3, 2\alpha + 1)$, $(2\alpha - 1, -2\alpha - 1)$ \\
\hline
\end{longtable}
}
In particular, if $(d,D) \in \ZZ/\ZZ^* \times \ZZ/\ZZ^2$ does not appear in the following tables, then $\cA\cP_{(1,1,1,d)} \left(\QQ( \sqrt{D} )\right) = \cA\cP_{(1,1,1,d)} (\QQ)$:

\begin{longtable}{|c|c|c|c|c|c|c|c|c|c|c|c|}
\hline
$d$ & \multicolumn{11}{|c|}{$1$} \\
\hline
$D$ & $-1$ & $2$ & $3$ & $5$ & $6 $& $11$& $14$ & $17$ & $21$ & $29$ & $41$ \\
\hline
$\# \mathcal{AP}_{(1,1,1,d)} (\QQ(\sqrt{D}))$ &$16$ & $12$& $6$& $26$& $8$& $8$& $6$& $8$& $8$& $12$& $12$ \\
\hline
\end{longtable}
\vspace*{-4mm}
\begin{longtable}{|c|c|c|c|c|c|c|c|c|c|c|c|c|c|}
\hline
$d$ & \multicolumn{5}{|c|}{$2$} & \multicolumn{3}{|c|}{$3$}  & $4$ & $6$ & $7$ & $9$ & $11$\\
\hline
$D$ & $-1$ & $2$ & $6$ & $11$ & $14$ & $3$ & $6$ & $21$ & $2$ & $6$ & $2$ & $3$  & $5$ \\
\hline
$\# \mathcal{AP}_{(1,1,1,d)} (\QQ(\sqrt{D}))$ & $5$& $3$& $3$& $5$& $3$& $6$& $8$& $8$& $3$& $3$& $3$& $3$& $3$ \\
\hline
\end{longtable}

\end{conjecture}

\begin{remark}
Denote by $\mathcal{F}_2$ the union of all quadratic fields, and for $d\in\ZZ$ define the plane curve
$$
\cD_d: dY^3-3Y^2Z-dX^2Y-2X^2Z=0.
$$
Corvaja and Zannier \cite[Corollary 1]{CZ} tell us that $\cD_d(\mathcal{O}_{\mathcal{F}_2})$ is finite. The above conjecture asserts that in fact we have been able to compute explicitely the set $\cD_d(\mathcal{O}_{\mathcal{F}_2})$. Moreover, it asserts that the uniparametric family $\cD_d$, where $d$ run over the rational intergers,  has only a finite number of points over $\mathcal{O}_{\mathcal{F}_2}$ and gives all of them.
\end{remark}

After that, we have taken a longer step, and have computed many examples for $\mathcal{AP}_{(1,1,1,3)} (K)$ where $K$ is any number field such that $|\Delta_K|\le 10^4$. In particular if $|\Delta_K|\le 10^4$ then $[K:\QQ]\le 7$, after Minkowski's bound. We have used the online tables of number fields with bounded discriminant from the PARI group \cite{pari}. There are precisely $9115$ such fields.

In the table below we display the minimal polynomial for a primitive element of $K$, along with the discriminant $\Delta_K$ and the number $n_K$ of Markoff triples in a.p. over $\mathcal{O}_K$ (for those cases where there are more than the known rational triples), that is $n_K=\# \mathcal{AP}_{(1,1,1,3)} (K)$.  
{\small 
\begin{longtable}{|c|c|c||c|c|c|}
\hline
K & $\Delta_K$ & $n_K$ & K & $\Delta_K$ & $n_K$\\
\hline
$x^2 - 3$ & $12$ & $6$ & $x^2 - x - 5$ & $21$ & $8$\\ \hline 
$x^2 - 6$ & $24$ & $8$ & $x^3 - 2$ & $-108$ & $6$\\ \hline 
$x^3 - x^2 - 8x - 3$ & $1425$ & $6$ & $x^3 - x^2 - 6x + 3$ & $993$ & $6$\\ \hline 
$x^3 - x^2 - 8x + 9$ & $1257$ & $8$ & $x^3 - x^2 - 7x + 4$ & $1509$ & $6$\\ \hline 
$x^3 - x^2 - 8x - 1$ & $1937$ & $6$ & $x^3 - 9x - 3$ & $2673$ & $6$\\ \hline 
$x^3 - x^2 - 13x + 1$ & $2292$ & $6$ & $x^3 - x^2 - 9x + 6$ & $3021$ & $6$\\ \hline 
$x^3 - 8x - 2$ & $1940$ & $6$ & $x^3 - 9x - 5$ & $2241$ & $6$\\ \hline 
$x^3 - 9x - 2$ & $2808$ & $6$ & $x^3 - 10x - 2$ & $3892$ & $6$\\ \hline 
$x^3 - x^2 - 10x + 7$ & $4065$ & $6$ & $x^3 - x^2 - 15x - 15$ & $3540$ & $6$\\ \hline 
$x^3 - x^2 - 18x + 33$ & $5073$ & $6$ & $x^3 - x^2 - 21x + 33$ & $5172$ & $6$\\ \hline 
$x^3 - x^2 - 13x - 10$ & $3877$ & $6$ & $x^3 - x^2 - 12x + 15$ & $4281$ & $5$\\ \hline 
$x^3 - 21x - 24$ & $5373$ & $6$ & $x^3 - 12x - 6$ & $5940$ & $6$\\ \hline 
$x^3 - x^2 - 23x + 39$ & $6108$ & $6$ & $x^3 - x^2 - 12x + 3$ & $7473$ & $6$\\ \hline 
$x^3 - x^2 - 17x + 27$ & $8628$ & $6$ & $x^3 - x^2 - 18x + 30$ & $9192$ & $6$\\ \hline 
$x^3 - 30x - 27$ & $9813$ & $6$ & $x^4 - x^2 + 1$ & $144$ & $6$\\ \hline 
$x^4 - x^3 - x^2 - 2x + 4$ & $441$ & $8$ & $x^4 - 2x^2 + 4$ & $576$ & $8$\\ \hline 
$x^4 + 4x^2 + 1$ & $2304$ & $6$ & $x^4 + 9$ & $2304$ & $8$\\ \hline 
$x^4 + x^2 + 4$ & $3600$ & $6$ & $x^4 + 6x^2 + 18$ & $4608$ & $8$\\ \hline 
$x^4 + 3x^2 - 6x + 6$ & $4752$ & $6$ & $x^4 + 11x^2 + 25$ & $7056$ & $8$\\ \hline 
$x^4 - 2x^3 - x^2 + 2x + 22$ & $7056$ & $6$ & $x^4 - 2x^3 + 4x^2 + 6$ & $7488$ & $6$\\ \hline 
$x^4 - 2x^3 - 2x + 1$ & $-1728$ & $6$ & $x^4 - x^3 - 3x^2 - x + 1$ & $-1323$ & $8$\\ \hline 
$x^4 - x^3 - x^2 - 2x + 1$ & $-1791$ & $5$ & $x^4 - 5$ & $-2000$ & $6$\\ \hline 
$x^4 - x^2 - 3x - 2$ & $-2151$ & $5$ & $x^4 - x^3 - x^2 - 5x - 5$ & $-2475$ & $6$\\ \hline 
$x^4 - x^3 - x - 2$ & $-2943$ & $6$ & $x^4 - 3x^2 - 1$ & $-2704$ & $6$\\ \hline 
$x^4 - 4x^2 - 3x + 1$ & $-2763$ & $6$ & $x^4 - 2x^2 - 4$ & $-1600$ & $6$\\ \hline 
$x^4 - x^2 - 3x + 1$ & $-3303$ & $6$ & $x^4 + x^2 - 6x + 1$ & $-3312$ & $6$\\ \hline 
$x^4 - 2x^3 - x^2 + 2x - 2$ & $-3312$ & $6$ & $x^4 - 2x^3 - 2x + 2$ & $-3632$ & $6$\\ \hline 
$x^4 + 3x^2 - 9$ & $-3600$ & $10$ & $x^4 - x^3 + 2x^2 + x - 2$ & $-3951$ & $6$\\ \hline 
$x^4 - 2x^2 - 2$ & $-4608$ & $8$ & $x^4 + 2x^2 - 2$ & $-4608$ & $6$\\ \hline 
$x^4 - 2x^3 + 3x^2 + x - 2$ & $-4671$ & $5$ & $x^4 - 2x^3 - 3x - 1$ & $-4675$ & $6$\\ \hline 
$x^4 - 3x^2 - 6x - 3$ & $-5616$ & $6$ & $x^4 - 2x^3 + 3x^2 - 2x - 2$ & $-5616$ & $6$\\ \hline 
$x^4 + 2x^2 - 11$ & $-6336$ & $8$ & $x^4 - 2x^2 - 11$ & $-6336$ & $6$\\ \hline 
$x^4 - x^3 + 2x^2 - 2x - 2$ & $-6444$ & $6$ & $x^4 - 2x^3 + x^2 - 3$ & $-6768$ & $6$\\ \hline 
$x^4 - x^2 - 6x - 2$ & $-6768$ & $6$ & $x^4 - 3$ & $-6912$ & $6$\\ \hline 
$x^4 - x^3 - 4x - 5$ & $-6507$ & $6$ & $x^4 - 2x^2 - 3x + 3$ & $-6603$ & $6$\\ \hline 
$x^4 - x^3 - 3x^2 + 4x + 2$ & $-7668$ & $6$ & $x^4 - x^3 - x^2 + 10x - 20$ & $-6975$ & $7$\\ \hline 
$x^4 - x^2 - 3$ & $-8112$ & $6$ & $x^4 - x^3 + x^2 + 3x - 3$ & $-8739$ & $5$\\ \hline 
$x^4 - 2x^3 - 3x^2 + 4x - 2$ & $-8640$ & $8$ & $x^4 - 2x^3 - 3x^2 - 2x + 1$ & $-8640$ & $8$\\ \hline 
$x^4 - x^3 - x^2 - 4x - 2$ & $-9012$ & $6$ & $x^4 - x^3 + 5x^2 + x - 2$ & $-9036$ & $6$\\ \hline 
$x^4 - 2x^3 - 4x^2 - 2x + 1$ & $-9408$ & $6$ & $x^4 + 3x^2 - 12$ & $-9747$ & $7$\\ \hline 
$x^4 - x^3 - 4x^2 + 4x + 1$ & $1125$ & $10$ & $x^4 - 6x^2 + 4$ & $1600$ & $6$\\ \hline 
$x^4 - 2x^3 - 7x^2 + 8x + 1$ & $3600$ & $6$ & $x^4 - 4x^2 + 1$ & $2304$ & $24$\\ \hline 
$x^4 - 2x^3 - 3x^2 + 4x + 1$ & $4752$ & $6$ & $x^4 - 2x^3 - 4x^2 + 5x + 5$ & $2525$ & $8$\\ \hline 
$x^4 - 2x^3 - 4x^2 + 2x + 1$ & $7488$ & $6$ & $x^4 - 5x^2 + 1$ & $7056$ & $12$\\ \hline 
$x^4 - x^3 - 7x^2 + 3x + 9$ & $4525$ & $7$ & $x^4 - 2x^3 - 7x^2 + 2x + 7$ & $9792$ & $6$\\ \hline 
$x^4 - 6x^2 - 3x + 3$ & $9909$ & $10$ & $x^4 - x^3 - 5x^2 + 3x + 4$ & $8468$ & $6$\\ \hline 
$x^4 - 5x^2 + 2$ & $9248$ & $6$ & $x^5 - 2x^4 + x^2 - 2x - 1$ & $-9759$ & $6$\\ \hline 
\end{longtable}
}
The behaviour seems rather unpredictable here, including some instances where some fields do appear, but none of its subfields do (like the example with $\Delta_K = 1125$ above). And the existence of many examples close to the chosen bound prevents us to establish a conjecture for this case, as we did above.

\begin{remark} Here we present the results concerning the density of number fields of bounded discriminant which have non--rational Markoff triples in a.p.; let $\Delta \in \NN$ and  
$$
r_\Delta = \frac{ \displaystyle \#  \left\{ K \; | \; |\Delta_K|\le \Delta, \;   \cA\cP_{(1,1,1,3)} (\QQ)\subsetneq \cA\cP_{(1,1,1,3)}(K)  \right\} }{ \displaystyle \# \{ K \; | \; |\Delta_K| \leq \Delta \} }.
$$

Then

{
\begin{center}
\begin{tabular}{|c|c|c|c|c|c|c|}
\hline
$\Delta$ & $50$ & $100$ & $500$ & $1000$ & $5000$ & $10000$ \\
\hline
$r_\Delta$ & $0.088$ & $0.042$ & $0.015$ & $0.009$ & $0.012$ & $0.0010$ \\
\hline
\end{tabular}
\end{center}
}\end{remark}

\

\

{\bf Data:} All the \verb+Magma+ sources are available on the first author's webpage.


\bibliographystyle{amsplain}

\end{document}